\documentclass[12]{article}
\usepackage{amsmath,amssymb,amsthm,color}

\newtheorem{theorem}{Theorem}[section]
\newtheorem{corollary}{Corollary}[section]
\newtheorem{lemma}{Lemma}[section]

\newtheorem{proposition}{Proposition}[section]
\newtheorem{definition}{Definition}[section]
\newtheorem{example}{Example}[section]

\begin{document}

\begin{center}{\bf \LARGE On $\ast-$Reverse Derivable Maps}\\
\vspace{.2in}
{\bf Gurninder S. Sandhu}\\
{\it Department of Mathematics,\\ Patel Memorial National College, Rajpura 140401,\\ Punjab, India.}\\
e-mail: gurninder\_rs@pbi.ac.in\\

{\bf Bruno L. M. Ferreira}\\
{\it Technological Federal University of Paran\'{a},\\
Professora Laura Pacheco Bastos Avenue, 800,\\
85053-510, Guarapuava, Brazil.}\\
e-mail: brunoferreira@utfpr.edu.br\\
and\\
{\bf Deepak Kumar}\\
{\it Department of Mathematics,\\ Punjabi University, Patiala-147002,\\ Punjab, India.}\\
e-mail: deep\_math1@yahoo.com\\

\end{center}

\begin{abstract}
Let $R$ be a ring with involution containing a nontrivial symmetric idempotent element $e$. Let $\delta: R\rightarrow R$ be a mapping such that $\delta(ab)=\delta(b)a^{\ast}+b^{\ast}\delta(a)$ for all $a,b\in R$, we call $\delta$ a $\ast-$reverse derivable map on $R$. In this paper, our aim is to show that under some suitable restrictions imposed on $R$, every $\ast-$reverse derivable map of $R$ is additive.
\end{abstract}

{\bf 2010 Mathematics Subject Classification.} 17C27 \\
{\bf Keyword:} Additivity, Reverse derivable maps, Involution, Peirce decomposition.

\section{Introduction}
Let $R$ be a ring, by a {\it derivation} of $R,$ we mean an additive map $\delta:R\rightarrow R$ such that $\delta(ab)=\delta(a)b+a\delta(b)$ for all $a,b\in R.$ A derivation which is not necessarily additive is said to be a {\it multiplicative derivation} or {\it derivable map} of $R.$ A mapping $\delta:R\rightarrow R$ is known as multiplicative Jordan derivation of $R$ if $\delta(ab+ba)=\delta(a)b+a\delta(b)+\delta(b)a+b\delta(a)$ for all $a,b\in R.$ In addition, $\delta$ is called $n-$multiplicative derivation of $R$ if $\delta(a_{1} a_{2}\cdots a_{n})=\sum_{i=1}^{n}a_{1}a_{2}\cdots\delta(a_{i})\cdots a_{n}$ for all $a_{1},a_{2},\cdots,a_{n}\in R.$ In \cite{RefJ2}, Herstein introduced a mapping ``$^\dag$" satisfying $(a+b)^{\dag}=a^{\dag}+b^{\dag}$ and $(ab)^{\dag}=b^{\dag}a+ba^{\dag}$ called a {\it reverse derivation}, which is certainly not a derivation. Moreover, a mapping $\delta: R\rightarrow R$ satisfying $\delta(ab)=\delta(b)a+b\delta(a)$ for all $a,b\in R$ is called a {\it multiplicative reverse derivation} or {\it reverse derivable map} of $R$. Let $e$ be an idempotent element of $R$ such that $e\neq 0, 1$. Then $R$ can be decomposed as follows:
\begin{center}$R=eRe\bigoplus eR(1-e)\bigoplus(1-e)Re\bigoplus(1-e)R(1-e)$\end{center}
This decomposition of $R$ is called {\it two-sided Peirce decomposition relative to $e$} (\cite{RefB1}, see pg. 48). It is easy to see that the components of this decomposition are the subrings of $R$ and for our convenience, we denote $R_{11}=eRe, R_{12}=eR(1-e), R_{21}=(1-e)Re$ and $R_{22}=(1-e)R(1-e)$. For any $r\in R,$ we denote the elements of $R_{ij}$ by $r_{ij}$ for all $i,j\in \{1,2\}.$ A mapping $\psi:R\rightarrow R$ is said to be a left (resp. right) centralizer if $\psi(ab)=\psi(a)b$ (resp. $\psi(ab)=a\psi(b)$) for all $a,b\in R.$ Moreover, if $\psi$ is left and right centralizer, then it is called {\it centralizer} of $R.$ A mapping $F:R\rightarrow R$ (not necessarily additive) such that $F(ab)=F(a)b+a\delta(b)$ for all $a,b\in R$ is said to be a multiplicative generalized derivation associated with derivation $\delta$ of $R.$ Note that, every multiplicative left centralizer is a multiplicative generalized derivation. By involution, we mean a mapping $\ast:R\rightarrow R$ such that $(x+y)^{\ast}=x^{\ast}+y^{\ast},$ $(x^{\ast})^{\ast}=x$ and $(xy)^{\ast}=y^{\ast}x^{\ast}$ for all $x,y\in R.$ An element $s\in R$ satisfying $s^{\ast} = s$ is called a \textit{symmetric element} of R.

\par The problem of when a multiplicative mapping must be additive has been studied by several authors. In this direction, Martindale \cite{RefJ3} gave a remarkable result. He discovered a set of conditions on $R$ such that every multiplicative isomorphism and anti-automorphism on $R$ is additive. In 1991, inspired by Martindale's work Daif \cite{RefJ1} extended these results to multiplicative derivations. He imposed same restrictions on $R$ and obtained the additivity of multiplicative derivations. In a very nice paper \cite{RefJ4}, Eremita and Ilisevic discussed the additivity of multiplicative left centralizers that are defined from $R$ into a bimodule $M$ over $R$ and gave a number of applications of the main result, that is stated as follows:
\par \emph{Let $R$ be a ring and $M$ be a bimodule over $R.$ Further, let $e_{1}\in R$
be a nontrivial idempotent (and $1-e_{1}=e_{2}$) such that for any $m\in M=\{m\in M: mZ(R)=(0)\},$ where $Z(R)$ denotes the center of $R,$
\begin{itemize}
\item[(i)] $e_{1}me_{1}Re_{2}=(0)$ implies $e_{1}me_{1}=0,$
\item[(ii)] $e_{1}me_{2}Re_{1}=(0)$ implies $e_{1}me_{2}=0,$
\item[(iii)] $e_{1}me_{2}Re_{2}=(0)$ implies $e_{1}me_{2}=0,$
\item[(iv)] $e_{2}me_{1}Re_{2}=(0)$ implies $e_{2}me_{1}=0,$
\item[(v)] $e_{2}me_{2}Re_{1}=(0)$ implies $e_{2}me_{2}=0,$
\item[(vi)] $e_{2}me_{2}Re_{2}=(0)$ implies $e_{2}me_{2}=0.$
\end{itemize}
Then every left centralizer $\phi:R\rightarrow M$ is additive}. An year later, Daif and Tammam-El-Sayiad \cite{RefJ5} investigated the additivity of multiplicative generalized derivations. In 2009, Wang \cite{RefJ6} extended the result of Daif and obtained the additivity of $n-$multiplicative derivation of $R.$ In a recent paper, Jing and Lu \cite{RefJ8} examined the additivity of multiplicative Jordan and multiplicative Jordan triple derivations. This sort of problems and their solutions are not limited only to the class associative rings. For the case of additivity of maps defined on non-associative rings and having a nontrivial idempotent, some results have already been proved. In alternative rings we can mention the works in \cite{bruno}, \cite{RBru}, \cite{RBru2}, \cite{BJH1}, \cite{BJH2}, \cite{BJH3}, \cite{RefJ7}, \cite{BR}. In light of all the cited papers, the natural question could be whether the results obtained for multiplicative derivations can also be discussed for multiplicative reverse derivations. In this paper, we consider this problem and answer it with the same set of assumptions taken by Martindale and Daif. Moreover, some appropriate examples are also given.

\section{Main Results}
\begin{definition}\label{definition1}
  Let $\ast$ be an involution on $R.$ Then an additive mapping $\delta:R\to R$ is called the $\ast-$reverse derivation if $\delta(ab)=\delta(b)a^{\ast}+b^{\ast}\delta(a)$ for all $a,b\in R.$ If $\delta$ is not necessarily additive then it is called $\ast-$reverse derivable map of $R.$
\end{definition}

\begin{example}\label{example1}
Let $R=\left\{\left(\begin{array}{ccc}
a & b \\
c & d
\end{array}
\right): a,b,c,d \in \mathbb{Z}\right\}$, where $\mathbb{Z}$ is the ring of integers. Define a mapping $\delta:R\rightarrow R$ such that $\delta\left(\begin{array}{ccc}
a & b \\
c & d
\end{array}
\right)=\left(\begin{array}{ccc}
-b & 2b \\
a-2c-d & b
\end{array}
\right)$ and $\ast$ be the standard involution of $R.$ Clearly, $\delta$ is a $\ast-$reverse derivation, which is neither a derivation nor a reverse derivation.
\end{example}

The main result of this paper reads as follows:

\begin{theorem}\label{theorem1}
Let $R$ be a ring with involution containing a nontrivial symmetric idempotent element $e$ such that
the following conditions are satisfied
\begin{itemize}
\item[(M1)] $xR=0$ implies $x=0.$

\item[(M2)] $eRx=0$ implies $x=0$ (hence $Rx=0$ implies $x=0$).

\item[(M3)] $exeR(1-e)=0$ implies $exe=0.$
 \end{itemize}
 Then every $\ast-$reverse derivable map $\delta: R\rightarrow R$ is additive.
\end{theorem}

\begin{corollary}\label{corollary7.2.1}
Let $R$ be a prime ring with a nontrivial symmetric idempotent $e.$ Then every $\ast-$reverse derivable map of $R$ is additive.
\end{corollary}

\begin{corollary}\label{corollary2}
Let $R$ be the ring same as in Theorem \ref{theorem1} and
\[
\mathcal{R}=\left\{\left(
                         \begin{array}{cc}
                           r_{11} & r_{12} \\
                           r_{21} & r_{22} \\
                         \end{array}
                       \right): r_{ij}\in R_{ij}\right\}\cong R_{11}\oplus R_{12}\oplus R_{21}\oplus R_{22}=R.
                       \]
Moreover,
                       $R_{11}\equiv\left\{\left(
                                 \begin{array}{cc}
                                   r_{11} & 0 \\
                                   0 & 0 \\
                                 \end{array}
                               \right): r_{11}\in R_{11}\right\}.$ Similarly to other spaces $R_{12},R_{21}$ and $R_{22}.$
                               Let $\mathcal{E}=\left(
                                                  \begin{array}{cc}
                                                    e & 0 \\
                                                    0 & 0 \\
                                                  \end{array}
                                                \right)$ be the non-trivial idempotent in $\mathcal{R}.$ Define $\delta:\mathcal{R}\to \mathcal{R}$ such that $\delta(XY)=\delta(Y)\tau(X)+\tau(Y)\delta(X)$ for all $X,Y\in\mathcal{R},$ where $\tau$ is the transpose map, which is named \textit{transpose reverse derivable map}. Under the same conditions (M1)-(M3) of Theorem \ref{theorem1}, every transpose reverse derivable map is additive.
\end{corollary}

It is easy to note that $\delta(e)=a_{11}+a_{12}+a_{21}+a_{22}.$ Since $\delta(e)=\delta(e^{2})=\delta(e)e^{\ast}+e^{\ast}\delta(e),$ it follows that $\delta(e)=a_{12}+a_{21}.$ Define a mapping $\wp:R\to R$ such that $\wp(x)=[a_{21}-a_{12},x^{\ast}].$ It is not difficult to check that $\wp$ is an additive $\ast-$reverse derivable map. Thus, we set $\Delta=\delta-\wp,$ which is also a $\ast-$reverse derivable map and $\Delta$ is additive if and only if $\delta$ is so. Moreover it is easy to observe that $\Delta(e)=0$.

We shall use the following fact very frequently in the sequel.
\begin{proposition}\label{proposition1}
Let $s\in R$ ($s_{ij}\in R_{ij},$ where $i,j\in\{1,2\}$). Then $s_{ij}^{\ast}=r_{ji},$ where $r=s^{\ast}\in R.$ Moreover, $s_{ij}=r_{ji}^{\ast}.$
\end{proposition}
\begin{proof}
Let $s\in R$ be any element. Then for $es(1-e)=s_{12}\in R_{12},$ we have $(es(1-e))^{\ast}=(1-e)^{\ast}s^{\ast}e^{\ast}=(1-e)s^{\ast}e.$ It gives that $s_{12}^{\ast}=r_{21},$ where $r=s^{\ast}.$ Similarly, one can easily observe that $s_{21}^{\ast}=r_{12},~s_{11}^{\ast}=r_{11}$ and $s_{22}^{\ast}=r_{22}.$ Moreover, for each $s_{ij}\in R$ there exists unique $r\in R$ such that $r_{ji}^{\ast}=s_{ij}$ as $\ast$ is bijective.
\end{proof}

\begin{lemma}\label{lemma1}
$\Delta(0)=0.$
\end{lemma}
\begin{proof}
The proof is trivial.
\end{proof}
\begin{lemma}\label{lemma2}$\Delta(R_{ij})\subset R_{ji}$, where $i,j=\left\{1,2\right\}$.
\end{lemma}
\begin{proof}
For any $x_{11}\in R_{11}$, we have $\Delta(x_{11})=\Delta(ex_{11}e)= \Delta(x_{11}e)e^{\ast} =e^{\ast}\Delta(x_{11})e^{\ast}=e\Delta(x_{11})e\in R_{11}$. Hence $\Delta(R_{11})\subset R_{11}$.
\\For any $x_{22}\in R_{22}$, $\Delta(x_{22})\in R$, we put $\Delta(x_{22})=r_{11}+r_{12}+ r_{21}+r_{22}$. Now $0=\Delta(ex_{22})=\Delta(x_{22})e^{\ast}=(r_{11}+r_{12}+ r_{21}+r_{22})e=r_{11}+r_{21}$. Likewise $0=\Delta(x_{22}e)=e^{\ast}\Delta(x_{22})=r_{11}+r_{12}$. It implies $r_{11}=r_{21}=r_{12}=0$. Therefore $\Delta(x_{22})=r_{22}$ and hence $\Delta(R_{22})\subset R_{22}$.
\\ For any $x_{12}\in R_{12}$, $\Delta(x_{12})=b_{11}+b_{12} +b_{21}+b_{22}$. Now $\Delta(x_{12})=\Delta(ex_{12})=\Delta(x_{12})e^{\ast}=b_{11}+b_{21}$ and $0=\Delta(x_{12}e)=e^{\ast}\Delta(x_{12})=e(b_{11}+ b_{12} + b_{21} + b_{22})=b_{11} + b_{12}$. Thus $\Delta(x_{12})= b_{21}$ and hence $\Delta(R_{12})\subset R_{21}$.
\\ Let be $x_{21}\in R_{21}$ then $\Delta(x_{21})=c_{11}+c_{12} +c_{21}+c_{22}$. Now $\Delta(x_{21})=\Delta(x_{21}e)=e^{\ast}\Delta(x_{21})=c_{11}+c_{12}$ and $0=\Delta(ex_{21})=\Delta(x_{21})e^{\ast}=(c_{11}+c_{12} + c_{21} + c_{22})e=c_{11} + c_{21}$. That yields $\Delta(x_{21})=c_{12}$ and hence $\Delta(R_{21})\subset R_{12}$.
\end{proof}
The following Lemmas has the same hypotheses of Theorem \ref{theorem1} and we need these Lemmas for the proof of the main result.

\begin{lemma}\label{lemma3}
$\Delta(x_{ii}+x_{jk})=\Delta(x_{ii})+\Delta(x_{jk})$ for all $x_{ii}\in R_{ii}$ and $x_{jk}\in R_{jk}.$
\end{lemma}
\begin{proof}
Firstly, we assume that $i=1=k$ and $j=2$. For any $r_{1n}\in R_{1n},$ where $n\in\{1,2\},$ we have
\begin{eqnarray*}
  (\Delta(x_{11})+\Delta(x_{21}))r_{1n} &=& \Delta(x_{11})r_{1n} \\
   &=& \Delta(x_{11})s_{n1}^{\ast} \\
   &=& \Delta(s_{n1}x_{11})-x_{11}^{\ast}\Delta(s_{n1}) \\
   &=& \Delta(s_{n1}(x_{11}+x_{21}))-x_{11}^{\ast}\Delta(s_{n1}) \\
   &=& \Delta(x_{11}+x_{21})s_{n1}^{\ast}+(x_{11}+x_{21})^{\ast}\Delta(s_{n1})-
   x_{11}^{\ast}\Delta(s_{n1})\\
   &=& \Delta(x_{11}+x_{12})r_{1n}.
\end{eqnarray*}
Therefore, we have
\begin{equation}\label{Eq-1}
(\Delta(x_{11})+\Delta(x_{21})-\Delta(x_{11}+x_{21}))r_{1n}=0\mbox{~~for~all~~}n\in\{1,2\}.
\end{equation}
For any $r_{2n}\in R_{2n},$ we have
\begin{eqnarray*}
  (\Delta(x_{11})+\Delta(x_{21}))r_{2n} &=& \Delta(x_{21})r_{2n} \\
   &=& \Delta(x_{21})s_{n2}^{\ast} \\
   &=& \Delta(s_{n2}x_{21})-x_{21}^{\ast}\Delta(s_{n2}) \\
   &=& \Delta(s_{n2}(x_{11}+x_{21}))-x_{21}^{\ast}\Delta(s_{n2}) \\
   &=& \Delta(x_{11}+x_{21})s_{n2}^{\ast}+(x_{11}+x_{21})^{\ast}\Delta(s_{n2})
   -x_{21}^{\ast}\Delta(s_{n2})\\
   &=& \Delta(x_{11}+x_{21})r_{2n}.
\end{eqnarray*}
That is
\begin{equation}\label{Eq-2}
(\Delta(x_{11})+\Delta(x_{21})-\Delta(x_{11}+x_{21}))r_{2n}=0\mbox{~~for~all~~}n\in\{1,2\}.
\end{equation}
Combining (\ref{Eq-1}) and (\ref{Eq-2}), we obtain
\[
(\Delta(x_{11})+\Delta(x_{21})-\Delta(x_{11}+x_{21}))R=0.
\]
By hypothesis (M1), we have
\[
\Delta(x_{11}+x_{21})=\Delta(x_{11})+\Delta(x_{21}).
\]
For the sake of completeness, now we consider the case when $i=j=1$ and $k=2.$ For any $r_{n1}\in R_{n1},$ we have
\begin{eqnarray*}
  r_{n1}(\Delta(x_{11})+\Delta(12)) &=& r_{n1}\Delta(x_{11}) \\
   &=& s_{1n}^{\ast}\Delta(x_{11}) \\
   &=& \Delta(x_{11}s_{1n})-\Delta(s_{1n})x_{11}^{\ast} \\
   &=& \Delta((x_{11}+x_{12})s_{1n})-\Delta(s_{1n})x_{11}^{\ast} \\
   &=& \Delta(s_{1n})(x_{11}+x_{12})^{\ast}+s_{1n}^{\ast}\Delta(x_{11}+x_{12})
   -\Delta(s_{1n})x_{11}^{\ast}\\
   &=& r_{n1}\Delta(x_{11}+x_{12}).
\end{eqnarray*}
Thus, we have
\begin{equation}\label{Eq-3}
r_{n1}(\Delta(x_{11})+\Delta(12)-\Delta(x_{11}+x_{12}))=0\mbox{~~for~all~~}n\in\{1,2\}.
\end{equation}
Likewise, for any $r_{n2}\in R_{n2},$ we obtain
\begin{equation}\label{Eq-4}
r_{n2}(\Delta(x_{11})+\Delta(12)-\Delta(x_{11}+x_{12}))=0\mbox{~~for~all~~}n\in\{1,2\}.
\end{equation}
Combining (\ref{Eq-3}) and (\ref{Eq-4}), we get
\[
R(\Delta(x_{11})+\Delta(x_{12})-\Delta(x_{11}+x_{12}))=0.
\]
By hypothesis (M2), we have
\[
\Delta(x_{11}+x_{12})=\Delta(x_{11})+\Delta(x_{12}).
\]
In the similar fashion, we can show that
\[
\Delta(x_{22}+x_{21})=\Delta(x_{22})+\Delta(x_{21}).
\]
and
\[
\Delta(x_{22}+x_{12})=\Delta(x_{22})+\Delta(x_{12}).
\]
\end{proof}

\begin{lemma}\label{lemma4}
$\Delta$ is additive on $R_{ij},$ where $i\neq j$ and $i,j\in\{1,2\}.$
\end{lemma}
\begin{proof}Assume that $i=1$ and $j=2.$ Let $x_{12},y_{12}\in R_{12}$ be any elements. For each $r_{n1}\in R_{n1},$ we note that
\[
r_{n1}(\Delta(x_{12})+\Delta(y_{12}))=0=r_{n1}\Delta(x_{12}+y_{12}).
\]
That gives
\begin{equation}\label{Eq-5}
r_{n1}(\Delta(x_{12})+\Delta(y_{12})-\Delta(x_{12}+y_{12}))=0\mbox{~~for~all~~}n\in\{1,2\}.
\end{equation}
Now, we observe that $(x_{12}+y_{12})s_{2n}=(e+x_{12})(s_{2n}+y_{12}s_{2n}).$ Thus we have
\begin{eqnarray*}
  \Delta((x_{12}+y_{12})s_{2n}) &=& \Delta((e+x_{12})(s_{2n}+y_{12}s_{2n})) \\
   &=& \Delta(s_{2n}+y_{12}s_{2n})(e+x_{12})^{\ast}+(s_{2n}+y_{12}s_{2n})^{\ast}\Delta(e+x_{12})\\
   &=& \Delta(s_{2n})x_{12}^{\ast}+\Delta(y_{12}s_{2n})+s_{2n}^{\ast}\Delta(x_{12}) \\
   &=& \Delta(x_{12}s_{2n})+\Delta(y_{12}s_{2n}).
\end{eqnarray*}
That is
\[
\Delta((x_{12}+y_{12})s_{2n})=\Delta(x_{12}s_{2n})+\Delta(y_{12}s_{2n}).
\]
Now, we consider
\begin{eqnarray*}
  r_{n2}(\Delta(x_{12})+\Delta(y_{12})) &=& s_{2n}^{\ast}(\Delta(x_{12})+\Delta(y_{12})) \\
   &=& s_{2n}^{\ast}\Delta(x_{12})+s_{2n}^{\ast}\Delta(y_{12})\\
   &=& \Delta(x_{12}s_{2n})-\Delta(s_{2n})x_{12}^{\ast}+\Delta(y_{12}s_{2n})-\Delta(s_{2n})y_{12}^{\ast} \\
   &=& \Delta(x_{12}s_{2n})+\Delta(y_{12}s_{2n})-\Delta(s_{2n})(x_{12}+y_{12})^{\ast} \\
   &=& \Delta(x_{12}s_{2n})+\Delta(y_{12}s_{2n})-\Delta((x_{12}+y_{12})s_{2n})+
   s_{2n}^{\ast}\Delta(x_{12}+y_{12})
\end{eqnarray*}
Using the above expression, we get
\begin{equation}\label{Eq-6}
 r_{n2}(\Delta(x_{12})+\Delta(y_{12})-\Delta(x_{12}+y_{12}))=0\mbox{~~for~all~~}n\in\{1,2\}.
\end{equation}
Combining (\ref{Eq-5}) and (\ref{Eq-6}), we find
\[
R(\Delta(x_{12})+\Delta(y_{12})-\Delta(x_{12}+y_{12}))=0.
\]
In view of assumption (M2), it follows that
\[
\Delta(x_{12})+\Delta(y_{12})=\Delta(x_{12}+y_{12}).
\]

\par We now discuss the case when $i=2$ and $j=1.$
Let $x_{21},y_{21}\in R_{21}$ be any elements. For each $r_{1n}\in R_{1n},$ we see that
\[
(\Delta(x_{21})+\Delta(y_{21}))r_{1n}=0=\Delta(x_{21}+y_{21})r_{1n}.
\]
That gives
\begin{equation}\label{Eq-7}
(\Delta(x_{21})+\Delta(y_{21})-\Delta(x_{21}+y_{21}))r_{1n}.\mbox{~~for~all~~}n\in\{1,2\}.
\end{equation}
One may observe that $s_{n2}(x_{21}+y_{21})=(s_{n2}+s_{n2}y_{21})(e+x_{21}).$ Thus we have
\begin{eqnarray*}
  \Delta(s_{n2}(x_{21}+y_{21})) &=& \Delta((s_{n2}+s_{n2}y_{21})(e+x_{21})) \\
   &=& \Delta(e+x_{21})(s_{n2}+s_{n2}y_{21})^{\ast}+(e+x_{21})^{\ast}\Delta(s_{n2}+s_{n2}y_{21})\\
   &=& \Delta(x_{21})s_{n2}^{\ast}+\Delta(s_{n2}y_{21})+x_{21}^{\ast}\Delta(s_{n2})\\
   &=& \Delta(s_{n2}x_{21})+\Delta(s_{n2}y_{21}).
\end{eqnarray*}
That is
\[
\Delta(s_{n2}(x_{21}+y_{21}))=\Delta(s_{n2}x_{21})+\Delta(s_{n2}y_{21}).
\]
Now, we consider
\begin{eqnarray*}
  (\Delta(x_{21})+\Delta(y_{21}))r_{2n} &=& \Delta(x_{21})r_{2n}+\Delta(y_{21})r_{2n} \\
   &=& \Delta(x_{21})s_{n2}^{\ast}+\Delta(y_{21})s_{n2}^{\ast}\\
   &=& \Delta(s_{n2}x_{21})-x_{21}^{\ast}\Delta(s_{n2})+\Delta(s_{n2}y_{21})
   -y_{21}^{\ast}\Delta(s_{n2})\\
   &=& \Delta(s_{n2}x_{21})+\Delta(s_{n2}y_{21})-(x_{21}+y_{21})^{\ast}\Delta(s_{n2}) \\
   &=& \Delta(s_{n2}x_{21})+\Delta(s_{n2}y_{21})-\Delta(s_{n2}(x_{21}+y_{21}))+
   \Delta(x_{21}+y_{21})s_{n2}^{\ast}
\end{eqnarray*}
Using the above expression, we get
\begin{equation}\label{Eq-8}
 (\Delta(x_{21})+\Delta(y_{21})-\Delta(x_{21}+y_{21}))r_{2n}=0\mbox{~~for~all~~}n\in\{1,2\}.
\end{equation}
Combining (\ref{Eq-7}) and (\ref{Eq-8}), we find
\[
(\Delta(x_{21})+\Delta(y_{21})-\Delta(x_{21}+y_{21}))R=0.
\]
By hypothesis (M1), we get
\[
\Delta(x_{21})+\Delta(y_{21})=\Delta(x_{21}+y_{21}).
\]
\end{proof}

\begin{lemma}\label{lemma5}
$\Delta$ is additive on $R_{11}.$
\end{lemma}
\begin{proof}
Let $x_{11},y_{11}\in R_{11}$ be any elements. For any $r_{12}\in R_{12},$ we find that
\begin{eqnarray*}
  (\Delta(x_{11})+\Delta(y_{11}))r_{12} &=& \Delta(x_{11})r_{12}+\Delta(y_{11})r_{12} \\
   &=& \Delta(x_{11})s_{21}^{\ast}+\Delta(y_{11})s_{21}^{\ast} \\
   &=& \Delta(s_{21}x_{11})+\Delta(s_{21}y_{11})-(x_{11}+y_{11})^{\ast}\Delta(s_{21}) \\
   &=&  \Delta(s_{21}x_{11})+\Delta(s_{21}y_{11})-\Delta(s_{21}(x_{11}+y_{11}))+\Delta(x_{11}
   +y_{11})s_{21}^{\ast}\\
   &=& \Delta(x_{11}+y_{11})r_{12}. ~~(using~Lemma~\ref{lemma4})
\end{eqnarray*}
It implies that
\[
(\Delta(x_{11})+\Delta(y_{11})-\Delta(x_{11}+y_{11}))R_{12}=0.
\]
In view of (M3), we obtain $\Delta(x_{11}+y_{11})=\Delta(x_{11})+\Delta(y_{11}).$

\end{proof}

\begin{lemma}\label{lemma6}
$\Delta$ is additive on $Re=R_{11}+R_{21}.$
\end{lemma}
\begin{proof}
Let $x_{11},y_{11}\in R_{11}$ and $x_{21},y_{21}\in R_{21}$ be any elements. We consider
\begin{eqnarray*}
  \Delta((x_{11}+x_{21})+(y_{11}+y_{21})) &=& \Delta((x_{11}+y_{11})+(x_{21}+y_{21})) \\
   &=& \Delta(x_{11}+y_{11})+\Delta(x_{21}+y_{21})~~(by~Lemma~\ref{lemma3}) \\
   &=& \Delta(x_{11})+\Delta(y_{11})+\Delta(x_{21})+\Delta(y_{21})\\
   && (by~Lemma~\ref{lemma4}~and~Lemma~\ref{lemma5}) \\
   &=& (\Delta(x_{11})+\Delta(x_{21}))+(\Delta(y_{11})+\Delta(y_{21})) \\
   &=& \Delta(x_{11}+x_{21})+\Delta(y_{11}+y_{21}) ~~(by~Lemma~\ref{lemma3}).
\end{eqnarray*}
It proves our claim.
\end{proof}

\textbf{Proof of Theorem \ref{theorem1}}: Let $r\in eR$ be any element. Then for each $x,y\in R,$ we have
\begin{eqnarray*}
  r(\Delta(x)+\Delta(y)) &=& r\Delta(x)+r\Delta(y) \\
   &=& s^{\ast}\Delta(x)+s^{\ast}\Delta(y) \\
   &=& \Delta(xs)+\Delta(ys)-\Delta(s)x^{\ast}-\Delta(s)y^{\ast} \\
   &=& \Delta((x+y)s)-\Delta(s)(x+y)^{\ast} ~~(using~Lemma~\ref{lemma6})\\
   &=& \Delta(s)(x+y)^{\ast}+s^{\ast}\Delta(x+y)-\Delta(s)(x+y)^{\ast}.\\
   &=& r\Delta(x+y)
\end{eqnarray*}
That yields, $eR(\Delta(x)+\Delta(y)-\Delta(x+y))=0.$
By hypothesis (M2), we find $\Delta(x+y)=\Delta(x)+\Delta(y).$ It completes the proof.
\\
We now construct an example to show that the restrictions imposed in Theorem \ref{theorem1} are sufficient (but not necessary) for the additivity of $\ast-$reverse derivable maps.

\begin{example}\label{example2}

Consider $R=\left\{\left(
                     \begin{array}{cc}
                       \overline{a} & \overline{b} \\
                       \overline{0} & \overline{a} \\
                     \end{array}
                   \right): a,b\in \mathbb{Z}_{6}\right\}$. Note that $R$ is a ring with nontrivial idempotent element $e=\left(
                        \begin{array}{cc}
                          \overline{3} & \overline{0} \\
                          \overline{0} & \overline{3} \\
                        \end{array}
                      \right)
                   $. There exists $0\neq \left(
                                                                     \begin{array}{cc}
                                                                       \overline{2} & \overline{4} \\
                                                                       \overline{0} & \overline{2} \\
                                                                     \end{array}
                                                                   \right)
                   =x\in R$ such that $eRx=0$ and mappings $\delta:R\rightarrow R$ defined by $\delta\left(
                                                                            \begin{array}{cc}
                                                                              \overline{a} & \overline{b} \\
                                                                              \overline{0} & \overline{a} \\
                                                                            \end{array}
                                                                          \right)=\left(
                                                                                    \begin{array}{cc}
                                                                                      \overline{0} & \overline{b} \\
                                                                                      \overline{0} & \overline{0} \\
                                                                                    \end{array}
                                                                                  \right)
                   $ is an additive $\ast-$reverse derivable map associated with the involution $\ast$ defined by $\left(
                                                                            \begin{array}{cc}
                                                                              \overline{a} & \overline{b} \\
                                                                              \overline{0} & \overline{a} \\
                                                                            \end{array}
                                                                          \right)^{\ast}=\left(
                                                                                    \begin{array}{cc}
                                                                                      \overline{a} & -\overline{b} \\
                                                                                      \overline{0} & \overline{a} \\
                                                                                    \end{array}
                                                                                  \right)
$. Thus, we see that our hypothesis is not satisfied but there exists a $\ast-$reverse derivable map which is additive.
\end{example}


\begin{thebibliography}{99}

\bibitem{RefJ1}
M. N. Daif, When is a multiplicative derivation additive?, Internat. J. Math. Math. Sci., 14(3) (1991), 615--618.

\bibitem{RefJ5}
M. N. Daif, M. S. Tammam-El-Sayiad, Generalized derivations which are additive, East-West J. Math., 9(1) (2007), 31--37.

\bibitem{RefJ4}
D. Eremita, D. Ilisevic, On additivity of centralisers, Bull. Aust. Math. Soc., 74 (2006), 177--184.


\bibitem{bruno} B.L.M. Ferreira, Jordan elementary maps on alternative division rings, Serdica Math. Journal, 43 (2017), 161--168.

\bibitem{RBru} R.N. Ferreira, B.L.M. Ferreira, Jordan derivation on alternative rings, International Journal of Mathematics, Game Theory, and Algebra, 25 (2017), 435--444

\bibitem{RBru2} R.N. Ferreira, B.L.M. Ferreira, Jordan triple derivation on alternative rings, Proyecciones Journal of Mathematics, 37 (2018), 169--178.


\bibitem{BJH1} B.L.M. Ferreira, J.C.M. Ferreira, H. Guzzo Jr., Jordan maps on alternative algebras, JP Journal of Algebra, Number Theory and Applications, 31 (2013), 129--142.

\bibitem{BJH2} B.L.M. Ferreira, J.C.M. Ferreira, H. Guzzo Jr., Jordan Triple Elementary Maps on Alternative Rings, Extracta Mathematicae, 29 (2014), 1--18.

\bibitem{BJH3} B.L.M. Ferreira, J.C.M. Ferreira, H. Guzzo Jr., Jordan Triple Maps of Alternative algebras, JP Journal of Algebra, Number Theory and Applications, 33 (2014), 25--33.

\bibitem{RefJ7}
J. M. Ferreira, B. L. M. Ferreira, Additivity of $n-$multiplicative maps on alternative rings, Commun. Algebra, 44 (2016), 1557--1568.

\bibitem{BR} B.L.M. Ferreira, R. Nascimento, Derivable maps on alternative rings, Recen, 16 (2014), 1--5.

\bibitem{RefJ2}
I. N. Herstein, Jordan derivations of prime rings, Proc. Amer. Math. Soc., 8 (1957), 1104--1110.

\bibitem{RefB1}
N. Jacobson, Structure of Rings, Amer. Math. Soc. Colloq. Pub., USA, 1964.
\bibitem{RefJ8}
W. Jing, F. Lu, Additivity of Jordan (triple) derivations on rings, Commun. Algebra, 40 (2012), 2700--2719.

\bibitem{RefJ3}
W. S. Martindale III, When are multiplicative mappings additive?, Proc. Amer. Math. Soc., 21 (1969), 695--698.

\bibitem{RefJ6}
Y. Wang, The additivity of multiplicative maps on rings, Commun. Algebra, 37 (2009), 2351--2356.


\end{thebibliography}
\end{document}